\documentclass[11pt]{amsart}
\usepackage{amsfonts,amsmath,amssymb,amsthm,cmtiup} 
\usepackage{mathrsfs}

\newtheorem*{theorem*}{Theorem}
\newtheorem{lemma}{Lemma} 
\newtheorem*{lemma*}{Lemma} 
 
\newtheorem*{corollary*}{Corollary}

\newcommand{\zero}{{\mathbf 0}}

\theoremstyle{definition} 

\newtheorem*{remark*}{Remark}

\newtheorem*{definition*}{Definition}

\newtheorem*{example*}{Example}

\begin{document}

\title[On the cardinality of Extremally Disconnected Groups]{On the cardinality\\ 
of Extremally Disconnected Groups\\ with Linear Topology} 
\author{Ol'ga Sipacheva}

\begin{abstract}
A group topology is said to be linear if open subgroups form a base of 
neighborhoods of the identity element. It is proved that the existence of a nondiscrete extremally 
disconnected group of Ulam nonmeasurable cardinality with linear topology implies that of a 
nondiscrete extremally disconnected group of cardinality at most $2^\omega$ with linear topology. 
\end{abstract}



\address{Department of General Topology and Geometry, Faculty of Mechanics and  Mathematics, 
M.~V.~Lomonosov Moscow State University, Leninskie Gory 1, Moscow, 199991 Russia}

\email{o-sipa@yandex.ru, osipa@gmail.com}

\maketitle

A topological space $X$ is said to be \emph{extremally disconnected} if the closure of any open set 
in $X$ is open. Extremally disconnected spaces have the following obvious properties:
\begin{enumerate}
\item
Any two disjoint open sets in an extremally disconnected space have disjoint closures. 
\item
Any open subset of an extremally disconnected space is extremally 
disconnected.
\item
The image of an extremally disconnected space under an open 
continuous map is extremally disconnected (this follows from the observation that, for any open 
continuous map $f\colon X\to Y$ and any $A\subset Y$, we have $\overline 
A=f(\overline{f^{-1}(A)})$).
\end{enumerate}
Yet another, not so obvious, property is the following theorem of Isbell~\cite{Isbell}:
\begin{enumerate}
\item[(4)]
A nondiscrete extremally space of Ulam 
nonmeasurable cardinality cannot be a $P$-space, i.e., must contain a nonopen $G_\delta$ 
set.
\end{enumerate} 
(Recall that a cardinal $\kappa$ is Ulam measurable if it carries a $\sigma$-complete (i.e., closed 
under countable intersections) nonprincipal ultrafilter on $\kappa$. The least Ulam measurable 
cardinal is measurable, so the nonexistence of Ulam measurable cardinals is consistent with ZFC, 
while the consistency of their existence cannot be derived from ZFC.)

The problem of the existence in ZFC of a nondiscrete extremally disconnected 
group~\cite{Arhangelskii} is still unsolved. However, it is known that the nonexistence of a 
countable nondiscrete extremally disconnected group is consistent with ZFC \cite{RS}. In this note, 
we prove that if there exists an extremally disconnected group of Ulam 
nonmeasurable cardinality which contains a countable family of open subgroups with nonopen 
intersection, then there exists a Boolean nondiscrete extremally disconnected group of cardinality 
at most~$2^\omega$. 

In what follows we use Malykhin's theorem that any extremally disconnected group 
contains an open Boolean subgroup~\cite{Malykhin}. 
Recall that a Boolean group is a group in which all nonidentity elements are of order~2. It is easy 
to see that all such groups are Abelian. Any Boolean group $B$ can be treated as a vector space over 
the two-element field $\mathbb F_2$ (and hence all Boolean groups of the same cardinality $\kappa$ 
are isomorphic to each other and to the direct sum $\bigoplus^\kappa \mathbb Z_2$  of $\kappa$ 
copies of the two-element group $\mathbb Z_2$). If $E$ is a basis of $B$, then $B$ is the set 
$[E]^{<\omega}$ of all finite subsets of $E$, and the group operation of $B$ is symmetric 
difference: $g+h=g\vartriangle h$ for $g,h\in B=[E]^{<\omega}$. The zero element of $B$ is the empty 
set. When considering the group operation (addition) on subsets of $B$ represented as 
$[E]^{<\omega}$, we use the standard notation $+$, while for the addition of elements, we often use 
the symmetric difference symbol $\vartriangle$. Thus, for $X,Y\subset B=[E]^{<\omega}$, 
$X+Y=\{x\vartriangle y: x\in X, y\in Y\}$.

\begin{theorem*}
Let $G$ be a Boolean extremally disconnected group with zero element $\zero$ containing a countable 
family of $\{H_n:n\in\omega\}$  of open subgroups such that $\bigcap H_n=\{\zero\}$. Then $G$ has an 
open subgroup which admits a continuous isomorphism onto a subgroup of the countable Cartesian 
product of discrete countable Boolean groups with the product topology. In particular, any such 
group $G$ contains an open subgroup of cardinality at most $2^\omega$.  
\end{theorem*}

\begin{lemma}
\label{lemma1}
If a topological group $G$ with identity element $1$ contains open subgroups $G=H_0\supset 
H_1\supset H_2\supset\dots$ such that $\bigcap H_n=\{1\}$, then $G$ admits a continuous isomorphism 
onto a subgroup of the Cartesian product $\prod D_n$ of discrete groups $D_n\cong G/H_n$ with the 
product topology. 
\end{lemma}

\begin{proof}
Since $\bigcap H_n=\{1\}$, the family $\mathscr H$ of natural homomorphisms $h_n\colon 
G\to G/H_n$ separates points of $G$. Each topological quotient $G/H_n$ is discrete (because the 
subgroups $H_n$ are open), and each $h_n$ is continuous. Therefore, the diagonal map $\Delta\mathscr 
H$ is a continuous monomorphism.
\end{proof}

\begin{lemma}
\label{lemma2}
Let $G$ be a Boolean topological group with zero $\zero$ which admits a continuous 
monomorphism $\varphi$ to the Cartesian product $\prod_{i\in \omega} D_i$ of discrete Boolean groups 
$D_i$ with the product topology, and let $\pi_n\colon \prod_{i\in \omega} D_i\to D_n$, $n\in 
\omega$, denote the canonical projections. Suppose that, for any neighborhood $U$ of $\zero$, there 
exists an $n$ such that $\pi_n(\varphi(U))$ is uncountable. Then there are two disjoint open sets 
$W_0$ and $W_1$ in $G$ such that $\zero\in \overline W_0\cap \overline W_1$. 
\end{lemma}

\begin{proof}
For convenience, we identify $G$ with a subgroup of $\prod_{i\in \omega} D_i$, i.e., assume 
$\varphi$ to be the identity monomorphism. 

Treating the groups $D_i$ as vector spaces over $\mathbb 
F_2$, we choose a basis $E_i$ in each $D_i$, so that $D_i=[E_i]^{<\omega}$. Given a $g\in D_i$, by 
$|g|$ we denote its cardinality as a finite subset of $E_i$. For convenience, we denote the zero
element of $D_i$ (the empty subset of $E_i$) by $\zero_i$.

For each $k\in \omega$ and each $\delta \in \{0,1\}$, we set 
\begin{multline*}
W^k_\delta=\{(g_n)_{n\in \omega}\in G\subset \prod_{i\in \omega} D_i: 
\text{$x_i=\zero_i$ for $i<k$, $x_k\ne \zero$},\\ 
\text{the 
number of 2's in the prime factorization of $|x_k|$}\\ 
\text{has the parity of $\delta$.}\} 
\end{multline*}
Note that $W^k_\delta$ is a union of sets of the form $\{\zero_0\}\times \dots\times 
\{\zero_{k-1}\}\times \{g\}\times \prod_{i>k}D_i\cap G$, each of which is open in $G$, because $D_k$ 
is discrete. Hence all $W^k_\delta$ are open. Clearly, $W^k_0\cap W^m_1=\varnothing$ for any $k,m\in 
\omega$. Therefore, $W_0=\bigcup_{k\in \omega}W_0^k$ and $W_1=\bigcup_{k\in \omega}W_1^k$ are 
disjoint open sets in $G$. Let us show that $\zero\in \overline W_0\cap \overline W_1$.

Take any neighborhood $U$ of $\zero$ in $G$. We must prove that $U\cap W_0\ne \varnothing$ and  
$U\cap W_1\ne \varnothing$. Let $V$ be a neighborhood of $\zero$ for which $8V\subset U$,  
and let $k$ be the least nonnegative integer for which $\pi_k(V)$ is uncountable. We have $V\subset 
\prod_{n\in \omega}\pi_n(V)$. Since all $\pi_i(V)$, $i<k$, are at most countable and the number of 
$k$-tuples of elements of countable sets is countable, it follows that there exists a 
$k$-tuple $(g_0, \dots, g_{k-1})\in \prod_{i<k}\pi_i(V)$ and an uncountable set $V'\subset V$ such 
that, for any $(x_n)_{n\in \omega},(y_n)_{n\in \omega}\in V'$,  
$x_i=y_i=g_i$ for all $i<k$ and $x_i\ne y_i$ for all $i\ge k$. Take any $\mathbf g
\in V'$ and let $W=\{\mathbf g+\mathbf x: \mathbf x\in V'\setminus \{\mathbf\}\}$. Then $W\subset 
V+V$, $\pi_k(W)$ is uncountable, $\pi_i(W)=\{\zero\}$ for all $i<k$, and $\zero_k\notin \pi_k(W)$. 

The $k$th coordinates of elements of $W$ are different elements of the Boolean group $D_k$, i.e., 
finite subsets of its basis $E_k$. Among these finite sets are uncouncountably many different sets 
of the same cardinality $m$. By the $\Delta$-system lemma, there exists a finite $R\subset E_k$ and 
an uncountable $\mathscr F\subset \pi_k(W)$ such that, for any $F,G\in \mathscr F$, we have 
$|F|=|G|=m$ and $F\cap G=R$. Note that all symmetric differences $F\vartriangle G$, $F,G\in \mathscr 
F$, are pairwise disjoint; moreover, for any different $F,G\in \mathscr 
F$, we have 
$$
F\vartriangle G\in\pi_k(2W)\subset \pi_k(4V)\quad\text{and}\quad |F\vartriangle 
G|=2(m-|R|)>0.
$$ 
Let us somehow split $\mathscr F$ into two disjoint subfamilies $\mathscr F'$ and 
$\mathscr F''$. For any different $F',G'\in \mathscr F'$ and any different  
$F'',G''\in \mathscr F''$, we have 
$$
F'\vartriangle 
G'\vartriangle F''\vartriangle G''\in\pi_k(4W)\subset \pi_k(8V)
$$
and
$$
|F'\vartriangle G'\vartriangle F''\vartriangle G''|=4(m-|R|). 
$$ 
Let $l$ be the number of 2's in the prime factorization of $m-|R|$. For any 
pairwise different $\mathbf x', \mathbf y',  \mathbf x'', \mathbf y''\in W$ such that $\pi_k(\mathbf 
x'), \pi_k(\mathbf y')\in \mathscr F'$ and  $\pi_k(\mathbf x''), \pi_k(\mathbf y'')\in \mathscr 
F''$, we have either 
$\mathbf x'+ 
\mathbf y'+\mathbf x''+ \mathbf y''\in W_0^k$ and $\mathbf x'+\mathbf y'\in W_1^k$ (if $l$ is even) 
or $\mathbf x'+ 
\mathbf y'+\mathbf x''+ \mathbf y''\in W_1^k$ and $\mathbf x'+\mathbf y'\in W_0^k$ (if $l$ is odd). 
In any case, $\mathbf x'+ \mathbf y'\in 2W\subset 4V\subset 8V \subset  U$ and  $\mathbf 
x'+ \mathbf y'+\mathbf x''+ \mathbf y''\in 4W\subset 8V \subset  U$. Thus, 
$U\cap W_\delta^k\ne \varnothing$ for $\delta=0,1$. 

We have shown that any neighborhood of $\zero$ intersects both sets $W_0$ and $W_1$. This means that 
$\zero\in \overline W_0\cap \overline W_1$.
\end{proof}

\begin{proof}[Proof of the theorem]
Clearly, we can assume the subgroups $H_n$ in the statement of the theorem to be decreasing. By 
Lemma~\ref{lemma1}, there exists a continuous monomorphism $G\to \prod_{n\in \omega} D_n$, 
where all $D_n$ are discrete Boolean groups, and by Lemma~\ref{lemma2} and property (1) of 
extremally disconnected spaces, $G$ contains a neighborhood $U$ of $\zero$ such that $|\pi_n(U)|\le 
\omega$ for all $n\in \omega$. Clearly, the subgroup $\langle U\rangle $ generated by $U$ in $G$ is 
open, and $\phi(\langle U\rangle)\subset \prod_{n\in \omega} \langle \pi_n(U)\rangle $, where each 
$\langle \pi_n(U)\rangle$ is the (at most countable) subgroup of $D_n$ generated by $\pi_n(U)$. 
\end{proof}

One of the immediate consequences of this theorem is concerned with  
extremally disconnected groups with linear topology. 

\begin{definition*}
A topological group is said to have \emph{linear topology} if its open subgroups form a base of 
neighborhoods of the identity element. 
\end{definition*}

All (consistent) examples of extremally disconnected groups known to the author have linear 
topology. Moreover, when the free Boolean topological group of a space with one nonisolated point is 
extremally disconnected, its free group topology turns out to be linear.

\begin{corollary*}
If there exists a nondiscrete extremally disconnected group of Ulam nonmeasurable 
cardinality with linear topology, then there exists a nondiscrete extemally disconnected Boolean 
group with linear topology admitting a continuous isomorphism to a subgroup of $\mathbb Z_2^\omega$ 
with the product topology. 
\end{corollary*}

\begin{proof}
Suppose that there exists a nondiscrete extremally disconnected group of Ulam 
nonmeasurable cardinality with linear topology. According to Malykhin's theorem mentioned above it 
has an open Boolean subgroup $B$. By property (2) of extremally disconnected spaces $B$ is 
extremally disconnected; obviously, it is linear and nondiscrete and has Ulam nonmeasurable 
cardinality. 

In view of property (4) of extremally disconnected spaces and the linearity of $B$, there exist open 
subgroups $J_n$, $n\in \omega$, of $B$ whose intersection is not open (but closed, since any open 
subgroup is closed). The topological quotient $G= B/\bigcap_{n\in \omega} J_n$ is nondiscrete, and 
it is extremally disconnected by property (3) (because any quotient homomorphism of topological 
groups is open).  Clearly, $G$ has linear topology, and its subgroups $H_n=J_n/\bigcap_{n\in \omega} 
J_n$ are open and satisfy the condition $\bigcap_{n\in \omega} H_n=\{\zero\}$ (here $\zero$ is the 
zero element of $G$). By the theorem $G$ has an open (and hence nondiscrete and extremally 
disconnected) subgroup $H$ for which there exists a continuous monomorphism $\varphi \colon H\to 
\prod_{n\in \omega}D_n$ from $H$ to the topological product of discrete countable Boolean groups 
$D_n$. Clearly, the induced topology of $H$ is linear. Each group $D_n$ is isomorphic to the 
countable subgroup $\bigoplus^\omega \mathbb Z_2$ of $\mathbb Z_2^\omega$; let $\psi_n$ denote the 
corresponding isomorphism. Since the discrete topology of $D_n$ is coarser than the topology unduced 
by the product topology, it follows that the Cartesian product $\psi=\prod_{n\in \omega}\psi_n\colon 
\prod_{n\in \omega} D_n\to \bigl(\bigoplus^\omega \mathbb Z_2\bigr)^\omega\subset \mathbb Z_2 
^\omega$ is a continuous monomorphism. Therefore, so is $\psi \circ \varphi\colon G\to \mathbb 
Z_2^\omega$. 
\end{proof}

\end{document}